\documentclass[12pt,reqno]{amsart}

\usepackage{amssymb}
\usepackage{amscd}
\usepackage{amsfonts,mathrsfs}
\usepackage{setspace}
\usepackage{version}

\usepackage{graphicx}

\newtheorem{theorem}{Theorem}[section]
\newtheorem{lemma}[theorem]{Lemma}
\newtheorem{proposition}[theorem]{Proposition}

\newtheorem*{multinil}{Theorem \cite[Theorem 8.6]{green-tao-nilratner}}
\newtheorem*{lem74rpt}{Lemma \cite[Lemma 7.4]{green-tao-nilratner}}
\newtheorem*{lem75rpt}{Lemma \cite[Lemma 7.5]{green-tao-nilratner}}

\theoremstyle{definition}

\theoremstyle{remark}

\renewcommand{\leq}{\leqslant}
\renewcommand{\geq}{\geqslant}
\newcommand\Lip{\operatorname{Lip}}
\newcommand\ab{\operatorname{ab}}

\newcommand\id{\operatorname{id}}
\newcommand\poly{\operatorname{poly}}

\newcommand{\md}[1]{\ensuremath{(\operatorname{mod}\, #1)}}

\newcommand\lin{{\operatorname{lin}}}

\newcommand\nonlin{{\operatorname{nlin}}}
\newcommand\X{\mathcal{X}}

\def\R{\mathbb{R}}
\def\C{\mathbb{C}}
\def\Z{\mathbb{Z}}
\def\E{\mathbb{E}}

\def\Q{\mathbb{Q}}

\def\eps{\varepsilon}

\numberwithin{equation}{section}

\begin{document}

\title[Erratum]{On the quantitative distribution of polynomial nilsequences -- erratum}


\author{Ben Green}
\address{Mathematical Institute\\
Radcliffe Observatory Quarter\\
Woodstock Road\\
Oxford OX2 6GG\\
England }
\email{ben.green@maths.ox.ac.uk}

\author{Terence Tao}
\address{Department of Mathematics, UCLA\\
405 Hilgard Ave\\
Los Angeles CA 90095\\
USA}
\email{tao@math.ucla.edu}



\begin{abstract}
This is an erratum to the paper \emph{The quantitative behaviour of polynomial orbits on nilmanifolds} by the authors, published as Ann. of Math. (2) \textbf{175} (2012), no. 2, 465--540. The proof of Theorem 8.6 of that paper, which claims  a distribution result for multiparameter polynomial sequences on nilmanifolds, was incorrect, and furthermore fails when at least one (but not all) of the $N_i$ are small. We provide two fixes for this issue here. First, we deduce the ``equal sides'' case $N_1 = \dots = N_t = N$ of this result from the 1-parameter results in the paper. This is the same basic mode of argument we attempted originally, though the details are different. The equal sides case is the only one required in applications such as the proof of the inverse conjectures for the Gowers norms due to the authors and Ziegler. Second, we sketch a proof that the multiparameter result in its original generality, that is to say without the equal sides restriction, does in fact hold, if one excludes the case when one of the $N_i$ are small. To obtain this statement the entire argument of our paper must be run in the context of multiparameter polynomial sequences $g : \Z^t \rightarrow G$ rather than 1-parameter sequences $g : \Z \rightarrow G$ as is currently done.
\end{abstract}

\maketitle
\tableofcontents
\section{Introduction}

We quote from \cite{green-tao-nilratner} and use its notation without any further comment. The problematic part of that paper is Section 8, in which a ``multiparameter quantitative Leibman theorem'', \cite[Theorem 8.6]{green-tao-nilratner} is established: results such as \cite[Theorem 1.19]{green-tao-nilratner} and \cite[Theorem 2.9]{green-tao-nilratner}, which involve only one variable polynomial maps, are not affected.

In \cite[Section 8]{green-tao-nilratner} we attempted to deduce a multiparameter result from the 1-parameter version, \cite[Theorem 2.9]{green-tao-nilratner}. Unfortunately the deduction is erroneous: the problem comes with the line ``By switching the indices $i_1,\dots, i_t$ if necessary\dots'' towards the end of the proof. The problem is that the horizontal character $\eta$ defined towards the start of the proof may change when this is done, and this invalidates the argument.  Furthermore, the theorem is in fact false in the case that one of the $N_i$ (but not all of them) is very small.  For instance, consider the polynomial map $g: \Z^2 \to \R/\Z$ given by $g(n_1,n_2) = \alpha (n_1-1) n_2$ for some irrational $\alpha$, and take $N_1=2$ and $N_2=N$ for a large $N$.  Then $g$ is highly non-equidistributed on $[N_1] \times [N_2]$, since it vanishes on the set $\{1\} \times [N_2]$, which is half of $[N_1] \times [N_2]$. However, it has an extremely large $C^\infty([N_1] \times [N_2])$ norm.

We thank Bryna Kra and Wenbo Sun for drawing this oversight to our attention, and for further drawing our attention to an error in the first version of this erratum, and to Marius Mirek for conversations that led to the counterexample when one of the $N_i$ is small.

Our aim is to correct these oversights. First, we deduce a multiparameter quantitative Leibman theorem from the 1-parameter version. However we are only able to do this in the ``equal parameters'' case of \cite[Theorem 8.6]{green-tao-nilratner} in which $N_1 = \dots = N_t = N$. To lift this restriction seems to require running the entire argument of \cite{green-tao-nilratner} in the context of multiparameter maps from $\Z^t$ to $G$. In \S \ref{rerun-sec} we provide a guide to doing this, of necessity extremely dependent on \cite{green-tao-nilratner}. The changes required in the multivariate case propagate right back to the most basic result in \cite{green-tao-nilratner}, Proposition 3.1, which must be proven in a multivariate setting.  To avoid the above counterexample, one has to add the following alternate conclusion to Theorem 8.6, namely that one has to also allow for the possibility that $N_i \ll \delta^{-O_{d,m,t}(1)}$ for some $i=1,\dots,t$.

The problematic result  \cite[Theorem 8.6]{green-tao-nilratner} was required in Sections 9 and 10 of \cite{green-tao-nilratner}, and as a consequence those results are restricted to the equal parameter case if one only uses the first fix contained in this erratum. By following \S \ref{rerun-sec}, one could remove this restriction, but now one has to add the hypothesis that $N_i \geq C \delta^{-C}$ for all $i=1,\dots,t$ and a sufficiently large $C$  depending on $d,m,t$ (or on $A,m,d$, in the case of Theorem 10.2).

Finally, in \S \ref{minor}, we list some additional minor errata to \cite{green-tao-nilratner}, which we take the opportunity to record here.

Let us briefly summarise the subsequent publications depending on \cite[Theorem 8.6]{green-tao-nilratner} that we are aware of.

\begin{itemize}
\item In \cite{gtz}, the proof of the $\mbox{GI}(s)$ conjectures, the appeal to \cite{green-tao-nilratner} occurs in Appendix D, specifically Theorem D.2. In this application we have $N_1 = \dots = N_t = N$.

\item In \cite{arithmetic-regularity}, the appeal to \cite{green-tao-nilratner} occurs in the proof of the counting lemma. A slightly modified version of the problematic \cite[Theorem 8.6]{green-tao-nilratner} is required, which is stated as \cite[Theorem 3.6]{arithmetic-regularity}. The proof of this is given in \cite[Appendix B]{arithmetic-regularity}, where it may be confirmed that again we only require the case $N_1 = \dots = N_t = N$. 

\item \cite{gorodnik-spatzier,gorodnik-spatzier2} Whilst these papers do state results depending on \cite[Theorem 8.6]{green-tao-nilratner} in which the equal sides condition is not assumed, the authors have confirmed to us that the main results of these papers, and in particular the results used subsequently in \cite{fks}, only require the equal sides case.
\end{itemize}

Let us recall the precise statement of \cite[Theorem 8.6]{green-tao-nilratner}, corrected as per the above discussion.  

\begin{multinil}
Let $0 < \delta < 1/2$, and let $m,t \geq 1$, $N_1,\dots, N_t \geq 1$ and $d \geq 1$ be integers. Write $\vec N = (N_1,\ldots,N_t)$ and $[\vec N] = [N_1] \times \dots \times [N_t]$.  Suppose that $G/\Gamma$ is an $m$-dimensional nilmanifold equipped with a $\frac{1}{\delta}$-rational Mal'cev basis $\mathcal{X}$ adapted to some filtration $G_{\bullet}$ of degree $d$, and that $g \in \poly(\Z^t, G_{\bullet})$. Then either $(g(\vec{n})\Gamma)_{\vec{n} \in [\vec{N}]}$ is $\delta$-equidistributed, or $N_i \ll \delta^{-O_{d,m,t}(1)}$ for some $i=1,\dots,t$, or else there is some horizontal character $\eta$ with $0 < \Vert \eta \Vert \ll \delta^{-O_{d,m,t}(1)}$ such that 
\[ \Vert \eta \circ g \Vert_{C^{\infty}[\vec{N}]} \ll \delta^{-O_{d,m,t}(1)}.\]
In the ``equal sides'' case $N_1=\dots=N_t$, the alternative $N_i \ll \delta^{-O_{d,m,t}(1)}$ may be deleted.
\end{multinil}

\textsc{Notation.} We will not explicitly indicate the dependence of constants $C$ or implied constants $O()$ on the parameters $m,t$ and $d$, which will remain fixed throughout this erratum. We will write $\mathscr{I}$ for the set of multi-indices $\vec{i} = (i_1,\dots, i_t)$ of total degree at most $d$, that is to say tuples of non-negative integers with $i_1 + \dots + i_t \leq d$. 

\section{Some results on polynomials}

In this section we record some useful distribution results on polynomials which we will need in both of the proofs of \cite[Theorem 8.6]{green-tao-nilratner}.  

We start with some remarks about Taylor coefficients and smoothness norms. If $f : \Z^t \rightarrow \R$ is a polynomial map then in \cite[Definition 8.2]{green-tao-nilratner} we defined the Taylor coefficients of $f$ by writing
\begin{equation}\label{taylor1} f(\vec{n}) = \sum_{\vec{i}} \alpha_{\vec{i}} \binom{\vec{n}}{\vec{i}}.\end{equation} We then defined the smoothness norm
\[ \Vert f \Vert_{C^{\infty}[\vec{N}]} := \sup_{\vec{i} \neq 0} \vec N^{\vec i} \Vert \alpha_{\vec{i}} \Vert_{\R/\Z}.\]
Here, however it is more convenient to use the conventional Taylor expansion
\begin{equation}\label{taylor2} f(\vec{n}) = \sum_{\vec{i}} \beta_{\vec{i}} \vec{n}^{\vec{i}},\end{equation} and to consider the variant smoothness norm
\[ \Vert f \Vert_{C_*^{\infty}[\vec{N}]} := \sup_{\vec{i} \neq 0} \vec N^{\vec i} \Vert \beta_{\vec{i}} \Vert_{\R/\Z}.\]
\begin{lemma}\label{taylor}
Suppose that $\Vert f \Vert_{C_*^{\infty}[\vec{N}]} \leq M$. Then there is some $r = O(1)$ such that $\Vert r f \Vert_{C^{\infty}[\vec{N}]} \ll M$.
\end{lemma}
\begin{proof}
This follows from the fact that $\alpha_{\vec{i}} = \sum_{\vec{j} \in \mathscr{I}} M_{\vec{i}, \vec{j}} \beta_{\vec{j}}$ with each $M_{\vec{i}, \vec{j}}$ rational with height $O(1)$ and $M_{\vec{i}, \vec{j}} = 0$ when $|\vec{j}| < |\vec{i}|$.
\end{proof}

We turn now to the following statement, which is actually the special case $G/\Gamma = \R/\Z$ of the problematic result \cite[Theorem 8.6]{green-tao-nilratner}.

\begin{proposition}\label{prop2.2}
Suppose that $g : \Z^t \rightarrow \R$ is a polynomial of total degree $d$, and let $0 < \delta < \frac{1}{2}$. Then either $(g(n) \md \Z)$ is $\delta$-equidistributed, or else there is some $q \in \Z$, $0 < |q| \ll \delta^{-O(1)}$, such that $\Vert q g \Vert_{C^{\infty}[\vec{N}]} \ll \delta^{-O(1)}$.
\end{proposition}
\begin{proof}
Slightly amusingly, the attempted argument of \cite[Theorem 8.6]{green-tao-nilratner}  is actually valid in this case. We run through the details briefly, referring the reader to the aforementioned argument if further clarification is required. 
A simple averaging argument confirms that, for $\gg \delta^{O(1)} N_2 \dots N_t$ values of $(n_2,\dots, n_t) \in [N_2 \times \dots \times N_t]$, the polynomial sequence $(g_{n_2,\dots, n_t}(n) \md{\Z})_{n \in [N_1]}$ is not $\delta^{O(1)}$-equidistributed, where $g_{n_2,\dots, n_t}(n) := g(n, n_2,\dots, n_t)$. For each such tuple, \cite[Theorem 2.9]{green-tao-nilratner} implies that there is an integer $\eta_{n_2,\dots, n_t}$ with $0 < | \eta_{n_2,\dots, n_t} | \ll \delta^{-O(1)}$ such that $\Vert \eta_{n_2,\dots, n_t}  g_{n_2,\dots, n_t} \Vert_{C^{\infty}[N_1]} \ll \delta^{-O(1)}$. By pigeonholing in the $\delta^{-O(1)}$ possible values of $\eta_{n_2,\dots, n_t}$ and passing to a thinner set of tuples $(n_2,\dots, n_t)$ we may assume that $\eta_{n_2,\dots, n_t} = \eta$ does not depend on $(n_2,\dots, n_t)$. Writing $p := \eta g$, and continuing to argue as in the proof of \cite[Theorem 8.6]{green-tao-nilratner} as far as (8.2), we deduce that for all $\vec{i}$ with $i_1 > 0$ there is some $q_{\vec{i}} \ll \delta^{-O(1)}$ such that $\Vert q_{\vec{i}} p_{\vec{i}} \Vert_{\R/\Z} \ll \delta^{-O(1)} / \vec N^{\vec i}$, where $p_{\vec{i}}$ is the $\vec{i}$th Taylor coefficient of $p$. Hence, defining $\tilde q_{\vec{i}} := \eta q_{\vec{i}}$, we have $\Vert \tilde q_{\vec{i}} g_{\vec{i}} \Vert_{\R/\Z} \ll \delta^{-O(1)} / \vec N^{\vec i}$. A similar argument holds whenever there is some index $j$ with $\vec{i}_j > 0$, that is to say whenever $\vec{i} \neq 0$. Taking $q := \prod_{\vec{i} \in \mathscr{I}} \tilde q_{\vec{i}}$, the result follows. (Note that in the attempted argument of \cite[Theorem 8.6]{green-tao-nilratner} we would obtain different horizontal characters $\eta_j$ for each $j$, which cannot be combined by simple multiplication to give a horizontal character independent of $j$ as we did here.)
\end{proof}

We use the above proposition to obtain a generalisation of \cite[Lemma 4.5]{green-tao-nilratner} to polynomials of several variables. 

\begin{proposition}\label{prop2.1}
Suppose that $g : \Z^t \rightarrow \R$ is a polynomial such that $\Vert g(\vec{n}) \Vert_{\R/\Z} \leq \eps$ for at least $\delta N_1 \ldots N_t$ values of $\vec{n} \in [\vec{N}]$, where $\eps < \delta/10$. Then there is some $Q \ll \delta^{-O(1)}$ such that $\Vert Q g \Vert_{C^{\infty}[\vec{N}]} \ll \delta^{-O(1)} \eps$. In particular, $\Vert Q g(\vec{0}) \Vert_{\R/\Z} \ll \delta^{-O(1)}\eps$.
\end{proposition}

\begin{proof}  This is essentially the same as the proof of \cite[Lemma 4.5]{green-tao-nilratner}. We include the argument for convenience. 
If $\eps \gg \delta^C$ then the result follows immediately from Proposition \ref{prop2.2}, so assume this is not the case. Expand 
\[ g(\vec{n}) = \sum_{\vec{i} \in \mathscr{I}} \alpha_{\vec{i}} \binom{\vec n}{\vec i} \] as a Taylor series. It follows from the assumption that none of the polynomials $\lambda g$, $\lambda \leq \delta/2\eps$, is $\delta^{O(1)}$-equidistributed on $[\vec{N}]$. Thus by Proposition \ref{prop2.1} we see that for each $\lambda \leq \delta/2\eps$ there is $q_{\lambda} \ll \delta^{-O(1)}$ such that $\Vert q_{\lambda} \lambda \alpha_{\vec{i}} \Vert_{\R/\Z} \ll \delta^{-O(1)} / \vec N^{\vec i}$ for all $\vec{i} \in \mathscr{I}$. Pigeonholing in the possible values of $q_{\lambda}$ we see that there is $q \ll \delta^{-O(1)}$ such that for $\gg \delta^{O(1)}/\eps$ values of $\lambda \leq \delta/2\eps$ we have $\Vert \lambda q \alpha_{\vec{i}} \Vert_{\R/\Z} \ll \delta^{-O(1)} / \vec N^{\vec i}$ for all  $\vec{i} \in \mathscr{I}$. It follows from \cite[Lemma 3.2]{green-tao-nilratner} that for each $\vec{i} \in \mathscr{I}$ there is $q_{\vec{i}} \ll \delta^{-O(1)}$ such that $\Vert q_{\vec{i}} \alpha_{\vec{i}} \Vert_{\R/\Z} \ll \eps \delta^{-O(1)} / \vec N^{\vec i}$. Writing $Q := \prod_{\vec{i} \in \mathscr{I}} q_{\vec{i}}$, we see that $Q \ll \delta^{-O(1)}$ and that $\Vert Q \alpha_{\vec{i}}\Vert_{\R/\Z} \ll \eps \delta^{-O(1)} / \vec N^{\vec i}$ for all $\vec{i} \in \mathscr{I}$. 

To get the final conclusion, note that 
\[ \Vert Q(g (\vec{n}) - g(\vec{0})) \Vert_{\R/\Z} \ll \delta^{-O(1)}\eps\] whenever $\vec{n} \in [\vec{N}]$. Since there is at least one value of $\vec{n}$ such that $\Vert g(\vec{n}) \Vert_{\R/\Z} \leq \eps$, and $Q \ll \delta^{-O(1)}$, the result follows.\end{proof}

We will need the following lemma of Schwartz-Zippel type.

\begin{lemma}[Schwartz-Zippel type lemma]\label{poly}
Let $f : \Z^t \rightarrow \R$ be a non-zero polynomial of degree $d$. Then the number of zeros of $f$ in $[L]^t \subset \Z^t$ is bounded by $O_{d,t}(L^{t-1})$.
\end{lemma}
\begin{proof}
We proceed by induction on $t$, the result being clear when $t = 1$. Expand
\[ f(n_1,\dots, n_t) = c_d(n_1,\dots, n_{t-1}) n_t^d + \dots + c_0(n_1,\dots, n_{t-1}).\]
For at least one value of $i$ the polynomial $c_i(n_1,\dots, n_{t-1})$ is not identically zero, and hence has $O_{d,t}(L^{t-2})$ roots $(n_1,\dots, n_{t-1}) \in [L]^{t-1}$ by the inductive hypothesis. However if $(n_1,\dots, n_{t-1})$ is not one of these roots then $f$ is nontrivial as a polynomial in $n_t$, and hence is satisfied by no more than $d$ values of $n_t$. 
\end{proof}

\section{Proof of \cite[Theorem 8.6]{green-tao-nilratner} in the case $N_1 = \dots = N_t$}

In this section we have $N_1=\ldots=N_t=N$.

Let $L$ be a positive integer parameter to be specified later (it will be $\delta^{-C}$ for some large $C$), and write $\vec L := (L,\ldots,L)$. Let the notation be as in \cite[Theorem 8.6]{green-tao-nilratner}, as repeated above.  The first step is to cover the cube $[N]^t$ by one-parameter progressions of length $N/L^2$ pointing in various directions.  More precisely, we have

\begin{lemma}\label{lemma1}
Suppose that $(g(\vec{n})\Gamma)_{\vec{n} \in [\vec{N}]}$ fails to be $\delta$-equidistributed. Suppose that $\vec{q}  \in [\vec{L}] = [L]^t$. Suppose that $N > L^2$ and that $L > C/\delta$ for some large $C$. Then $(g(\vec{x} + \vec{q} n)\Gamma)_{n \in [N/L^2]}$ fails to be $\frac{1}{2}\delta$-equidistributed for at least $\frac{1}{4}\delta N^t$ tuples $\vec{x} \in [\vec{N}]$.
\end{lemma}
\begin{proof}
Since $(g(\vec{n})\Gamma)_{\vec{n} \in [\vec{N}]}$ is not $\delta$-equidistributed, there is some Lipschitz function $F : G/\Gamma \rightarrow \C$, $\int_{G/\Gamma} F = 0$, such that 
\[ |\E_{\vec{n} \in [\vec{N}]} F(g(\vec{n}))| \geq \delta \Vert F \Vert_{\Lip}.\]
However, introducing an additional averaging, the left-hand side is equal to
\[ \E_{\vec{x} \in [\vec{N}]} \E_{n \in [N/L^2]} F(g(\vec{x} + \vec{q} n)) + O(\frac{1}{L}\Vert F \Vert_{\Lip}).\]
In particular if $L > C/\delta$ with $C$ large enough then we have
\[ \E_{\vec{x} \in [\vec{N}]} \E_{n \in [N/L^2]} F(g(\vec{x} + \vec{q} n)) \geq \frac{3}{4}\delta \Vert F \Vert_{\Lip}.\]
It follows that for at least $\frac{1}{4}\delta N^d$ tuples $\vec{x}$ we have
\[ \E_{n \in [N/L^2]} F(g(\vec{x} + \vec{q} n)) \geq \frac{1}{2}\delta\Vert F \Vert_{\Lip},\] and this implies the result.
\end{proof}

Write $p(\vec{n}) := \pi(g(\vec{n}))$, where $\pi$ is projection onto the horizontal torus $(G/\Gamma)_{\ab}$. Recall that the horizontal torus has dimension $m_{\ab}$, so $p$ takes values in $\R^{m_{\ab}}$. The total degree (highest degree of any monomial) of $p$ is at most $d$. Expand

\begin{equation}\label{eq0} p(\vec{x} + \vec{q} n) =\sum_{i=1}^d  \sum_{\vec{i} \in \mathscr{I}: |\vec{i}| = i} c_{\vec{i}}(\vec{x}) \vec{q}^{\vec{i}} n^{i}.\end{equation}
Here, the $c_{\vec{i}}:\Z^t \to \R^{m_{\ab}}$ are polynomials of total degree at most $d$.

Now we claim that the map from $\Z$ to $G$ defined by $n \mapsto g(\vec{x} + \vec{q} n)$ lies in $\poly(\Z, G_{\bullet})$. Indeed the map from $\Z^t$ to $G$ given by $\vec{n} \mapsto g(\vec{x} + \vec{q} \cdot \vec{n})$ lies in $\poly(\Z^t, G_{\bullet})$ by \cite[Corollary 6.8]{green-tao-nilratner}, and so it suffices to check that if $h(\vec{n}) \in \poly(\Z^t, G_{\bullet})$ then the diagonal map $h^{\Delta}(n) := h(n,n,\dots, n)$ lies in $\poly(\Z, G_{\bullet})$.  But this is obvious from the definition, \cite[Definition 6.1]{green-tao-nilratner}.

Suppose that $(g(\vec{x} + \vec{q}n)\Gamma)_{n \in [N/L^2]}$ fails to be $\frac{1}{2}\delta$-equidistributed. By Lemma \ref{lemma1}, for every $\vec{q} \in [\vec{L}]$ this is so for at least $\frac{1}{4}\delta N^t$ values of $\vec{x} \in [\vec{N}]$. By \cite[Theorem 2.9]{green-tao-nilratner}, which is applicable by the claim in the preceding paragraph, the following is therefore true. For all $\vec{q}$ with $0 \leq q_i < L$, there are at least $\frac{1}{4}\delta N^t$ choices of $\vec{x} \in [\vec{N}]$ such that there is some $\xi(\vec{q} , \vec{x}) \in \Z^{m_{\ab}}$, $0 < |\xi(\vec{q}, \vec{x})| \ll \delta^{-O(1)}$, such that 

\[ \big\Vert  \xi(\vec{q}, \vec{x}) \cdot \sum_{\vec{i} : |\vec{i}| = i}  c_{\vec{i}}(\vec{x}) \vec{q}^{\vec{i}}  \big\Vert_{\R/\Z} \ll (L/\delta)^{O(1)}N^{-i}\] for all $i = 1,\dots, d$. 

By the pigeonhole principle there is some $\xi \in \Z^{m_{\ab}}$, $0 < |\xi| \ll \delta^{-O(1)}$, together with a subset $\mathscr{S} \subset [\vec{L}] \times [\vec{N}]$, $|\mathscr{S}| \gg \delta^{O(1)} (LN)^t$, such that $\xi(\vec{q}, \vec{x}) = \xi$ whenever $(\vec{q}, \vec{x}) \in \mathscr{S}$. For each $\vec{q} \in [\vec{L}]$, write $X_q := \{ \vec{x} \in [\vec{N}] : (\vec{q}, \vec{x}) \in \mathscr{S}\}$. Then for $\gg \delta^{O(1)} L^t$ values of $\vec{q} \in [\vec{L}]$ we have $|X_q| \gg \delta^{O(1)} N^t$. Let $\mathscr{Q}$ be the set of such $\vec{q}$.

Thus
\begin{equation}\label{eq1} \big\Vert  \xi \cdot \sum_{\vec{i} : |\vec{i}| = i}  c_{\vec{i}}(\vec{x}) \vec{q}^{\vec{i}}\big\Vert_{\R/\Z} \ll (L/\delta)^{O(1)}N^{-i}\end{equation}whenever $\vec{x} \in X_{\vec{q}}$, and for all $i = 1,\dots, d$, and if $\vec{q} \in \mathscr{Q}$ then $|X_{\vec{q}}| \gg \delta^{C'} N^t$.

Now we apply Proposition \ref{prop2.1}, with $g : \Z^t \rightarrow \R$ given by
\[ g(\vec{x}) = g_{i, \vec{q}}(\vec{x})  := \xi \cdot \sum_{\vec{i} : |\vec{i}| = i}  c_{\vec{i}}(\vec{x}) \vec{q}^{\vec{i}}.\]
If $N > (L/\delta)^{C}$ for $C$ large enough then $\eps := (L/\delta)^{O(1)}N^{-i}$ is small enough that $\eps < \frac{1}{10}\delta^{C'}$ and so the proposition applies.

We conclude that for each $\vec{q} \in \mathscr{Q}$ and for each $i = 1,\dots, d$ there is some $Q_i = Q_i(\vec{q})$, $Q_i(\vec{q}) \ll \delta^{-O(1)}$, such that 

\[ \big\Vert  Q_i(\vec{q}) \xi \cdot \sum_{\vec{i} : |\vec{i}| = i} c_{\vec{i}}(\vec{0}) \vec{q}^{\vec{i}}\big\Vert_{\R/\Z} \ll (L/\delta)^{O(1)} N^{-i}.\]

Since there are $\ll \delta^{-O(1)}$ possibilities for $(Q_1(\vec{q}),\dots, Q_d(\vec{q}))$, we may pass to a set $\mathscr{Q}' \subset \mathscr{Q}$, $|\mathscr{Q}'| \gg \delta^{O(1)} L^d$, such that $Q_i(\vec{q}) = Q_i$ is independent of $\vec{q}$ as $\vec{q}$ ranges over $\mathscr{Q}'$. Setting $\tilde \xi := Q_1 \dots Q_d \xi$, we then have

\begin{equation}\label{eq2} \Vert \tilde\xi \cdot \sum_{\vec{i} : |\vec{i}| = i}  c _{\vec{i}}(\vec{0}) \vec{q}^{\vec{i}}\Vert_{\R/\Z} \ll (L/\delta)^{O(1)} N^{-i}.\end{equation} for all $\vec{q} \in \mathscr{Q}'$ and for all $i = 1,\dots,d$.

We claim that if $L = \delta^{-C}$ with $C$ big enough then as a consequence of \eqref{eq2} we have

\begin{equation}\label{eq7} \big\Vert \tilde{\tilde{\xi}}\cdot  c_{\vec{i}}(\vec{0}) \big\Vert_{\R/\Z} \ll  (L/\delta)^{O(1)}N^{-i} \ll \delta^{-O(1)} N^{-i}\end{equation} for all $\vec{i} \in \mathscr{I}$, where $\tilde{\tilde{\xi}} = \tilde Q \tilde \xi$ with $|\tilde Q| \ll \delta^{-O(1)}$.

Leaving the proof of this claim aside for the moment, setting $\vec{x} = \vec{0}$ and $n = 1$ in \eqref{eq0} reveals that 
\[ p(\vec{n}) = \sum_{\vec{i} \in \mathscr{I}} c_{\vec{i}}(\vec{0}) \vec{n}^{\vec{i}},\] and so \eqref{eq7} implies that $\Vert \tilde{\tilde{\xi}} \cdot p \Vert_{C_*^{\infty}[\vec{N}] }\ll \delta^{-O(1)}$. By Lemma \ref{taylor}, there is some $r = O(1)$ such that $\Vert r \tilde{\tilde{\xi}} \cdot p \Vert_{C^{\infty}[\vec{N}] }\ll \delta^{-O(1)}$.  Defining the horizontal character $\eta$ to be $r\tilde{\tilde{\xi}} \cdot \pi$, this concludes the proof of \cite[Theorem 8.6]{green-tao-nilratner} in the case $N_1 = \dots = N_t = N$.

It remains to check the claim \eqref{eq7}. We do this by taking linear combinations of \eqref{eq2} for different $\vec{q} \in \mathscr{Q}'$ in order to isolate each individual Taylor coefficient $c_{\vec{i}}(\vec{0})$. The key input is the following lemma.

\begin{lemma}\label{lemma3}
Let $\mathscr{Q} \subset [L]^t$ be a set of size $\eps L^t$, and to each $\vec{q} \in \mathscr{Q}$ associate the vector $v_{\vec{q}} := (\vec{q}^{\vec{i}})_{\vec{i} \in \mathscr{I}} \in \Q^{\mathscr{I}}$. Then, provided $L > C/\eps$, the $v_{\vec{q}}$ span $\Q^{\mathscr{I}}$.
\end{lemma}
\begin{proof}
If not, there is some $w \in \Q^{\mathscr{I}}$ such that $w \cdot v_{\vec{q}} = 0$ for all $\vec{q} \in \mathscr{Q}$. Thus
\[ \sum_{\vec{i}} w_{\vec{i}} \vec{q}^{\vec{i}} = 0\] whenever $\vec{q} \in \mathscr{Q}$. This is a polynomial equation of total degree $i$ in $q_1,\dots, q_t$, and it is not the trivial polynomial. Therefore by Lemma \ref{poly} this equation has $O(L^{t-1}) < |\mathscr{Q}|$ solutions, contrary to assumption.
\end{proof}

Returning to our proof of the claim \eqref{eq7}, take $L = \delta^{-C}$ large enough that Lemma \ref{lemma3} applies (with $\eps := |\mathscr{Q}'|/N^t$). Then for each $\vec{i} \in \mathscr{I}$ we may select $\vec{q}_1,\dots, \vec{q}_{|\mathscr{I}|} \in \mathscr{Q}'$ and rationals $\gamma_m$ such that 
\[ \mathbf{1}_{\vec{i} = \vec{j}}  = \sum_{m=1}^{|\mathscr{I}|} \gamma_m \vec{q}_m^{\vec{j}} .\]
Inverting these linear relations using the adjoint formula for the inverse (or by using Siegel's lemma), we see that the $\gamma_m$ are all rationals of height $\ll \delta^{-O(1)}$. Taking $\tilde Q$ to be the product of the denominators of all these $\gamma_m$, across all values of $\vec{i} \in \mathscr{I}$, we have $\tilde{Q} \ll \delta^{-O(1)}$ and now
\[ \tilde{Q} \mathbf{1}_{\vec{i} = \vec{j}} = \sum_{m = 1}^{|\mathscr{I}|} \gamma'_m \vec{q}_m^{\vec{j}}\] with the $\gamma'_m$ being \emph{integers} of size at most $\delta^{-O(1)}$. We may now take appropriate linear combinations of \eqref{eq2} to get the claim \eqref{eq7}, thereby concluding the argument.

\section{Removing the restriction $N_1 = \dots = N_t$}\label{rerun-sec}

As remarked in the introduction, we have been unable to deduce \cite[Theorem 8.6]{green-tao-nilratner} as stated (with the parameters $N_1,\dots, N_t$ not all equal) from the 1-parameter statement. To obtain this result we instead need to rerun the entire argument of the first seven sections of \cite{green-tao-nilratner} in the multiparameter setting. Unfortunately this makes an already complicated argument look even more fearsome. In this section we write $[\vec{N}] = [N_1] \times \dots \times [N_t]$.

Note first of all that the result follows from the following multidimensional version of \cite[Theorem 7.1]{green-tao-nilratner}.

\begin{theorem}[Variant of Main Theorem]\label{mainthm-ind}
Let $m_*,m,d,t \geq 0$ be integers with $m_* \leq m$. Let $0 < \delta < 1/2$ and suppose that $N_1,\ldots,N_t \geq 1$. Suppose that $G/\Gamma$ is a nilmanifold and that $G_{\bullet}$ is a filtration of degree $d$ and with nonlinearity degree $m_*$. Suppose that $\mathcal{X}$ is a $1/\delta$-rational Mal'cev basis adapted to $G_{\bullet}$ and suppose that $g \in \poly(\Z,G_{\bullet})$. If $(g(\vec n)\Gamma)_{\vec n \in [\vec N]}$ is not $\delta$-equidistributed then there is a horizontal character $\eta$ with $0 < |\eta| \ll \delta^{-O_{m,m_*,d,t}(1)}$ such that 
\[ \Vert \eta \circ g \Vert_{C^{\infty}[\vec N]} \ll \delta^{-O_{m,m_*,d,t}(1)}.\] 
\end{theorem}

As in \cite[\S 7]{green-tao-nilratner}, we prove this by induction on $d$ and $m_*$, assuming that the claim has already been proven for smaller $d$ (and arbitrary $m_*$), or for the same value of $d$ and a lesser value of $m_*$.

We allow all implied constants to depend on $m_*,d,t$.
By \cite[Lemma 3.7]{green-tao-nilratner} (which extends without difficulty to the multidimensional case), 
we may assume that the orbit $(g(\vec n)\Gamma)_{\vec n \in [\vec N]}$ is not $\delta^{O(1)}$-equidistributed along some vertical frequency $\xi \in \Z^{m_d}$ with $|\xi| \ll \delta^{-O(1)}$. Thus there is some function $F : G/\Gamma \rightarrow \C$ with $\Vert F \Vert_{\Lip} \leq 1$ and vertical frequency $\xi$ such that 
\begin{equation}\label{non-equi} |\E_{\vec n \in [\vec N]} F(g(\vec n)\Gamma) - \int_{G/\Gamma} F| \gg \delta^{O(1)}.\end{equation}

If $\xi = 0$ then $F$ is $G_d$-invariant and we may descend to $G/G_d$, together with the filtration $G_{\bullet}/G_d$ which has length $d-1$, and invoke our inductive hypothesis, exactly as in \cite[\S 7]{green-tao-nilratner}.  Thus we may assume that $\xi \neq 0$.  Since $F$ has $\xi$ as a vertical frequency, \eqref{non-equi} becomes
\begin{equation}\label{non-equi-2} |\E_{\vec{n} \in [\vec{N}]} F(g(\vec{n})\Gamma)| \gg \delta^{O(1)}.\end{equation}

Arguing exactly as in \cite[\S 7]{green-tao-nilratner}, we may reduce to the case $g(0)=\id_G$.  Next, we reduce to the case when $\psi(g(e_i)) \in  [0,1]^m$ for $i=1,\ldots,t$, where $e_1,\ldots,e_t$ is the standard basis for $\Z^t$.  To reduce to this case, we factorise $g(e_i) = \{g(e_i)\} [g(e_i)]$ as in \cite[Lemma A.14]{green-tao-nilratner}, then set
$$ \tilde g(\vec n) := g(\vec n) [g(e_1)]^{-n_1} \ldots [g(e_t)]^{-n_t}.$$
Then $\tilde g(\vec n)\Gamma = g(\vec n)\Gamma$, $\tilde g(0) = \id_G$, $\tilde g \in \poly(\Z^t,G_{\bullet})$ and $\pi(\tilde g(\vec n)\Gamma) = \pi(g(\vec n)\Gamma)$, so proving Theorem \ref{mainthm-ind} for $g$ is equivalent to proving it for $\tilde g$.  As $\tilde g(e_i) = \{ g(e_i) \}$, we have thus reduced to the case $\psi(g(e_i)) \in [0,1]^m$ for $i=1,\ldots,t$ as desired.

Henceforth we assume $g(0) = \id_G$ and $\psi(g(e_i)) \in [0,1]^m$ for $i=1,\ldots,t$.  We then apply \cite[Corollary 4.2]{green-tao-nilratner} (extended to higher dimensions in the obvious fashion) to deduce that for $\gg \delta^{O(1)}N$ values of $\vec h \in [-\vec N, \vec N]$, we have
\begin{equation}\label{eq801} |\E_{\vec n \in [\vec N]} F(g(\vec n+\vec h)\Gamma) \overline{F(g(\vec n) \Gamma)}| \gg \delta^{O(1)}.\end{equation}

We factor $g = g_\nonlin g_\lin$, where
$$ g_\lin(\vec n) := g(e_1)^{n_1} \ldots g(e_t)^{n_t}$$
is the \emph{linear part} of $g$, and
$$ g_\nonlin(\vec n) := g(\vec n) g_\lin(\vec n)^{-1}$$
is the \emph{nonlinear part} of $g$. Both maps are polynomial maps from $\Z^t$ to $G$; note that $g_\nonlin$ takes the values $1_G$ at $0$ and at $e_1,\ldots,e_t$, and so by Taylor expansion we may view $g_\nonlin$ as the product of finitely many functions of the form $\vec n \mapsto g_{\vec j}^{\binom{\vec n}{\vec j}}$ for some $\vec j$ with $|\vec j| \geq 2$, and some $g_{\vec j} \in G_{|\vec j|}$.  In particular, $g_\nonlin$ takes values in $G_2$.  We may then rewrite \eqref{eq801} as
\begin{equation}\label{use-soon-1} |\E_{\vec n \in [\vec N]} F_{\vec h}(g_{\vec h}(\vec n)\Gamma^2) | \gg \delta^{O(1)},\end{equation}
where $F_{\vec h}: G^2/\Gamma^2 \to \C$ is the function
\[ F_{\vec h}(x,y) := F(\{g_\lin(\vec h)\} x) \overline{F(y)}\] and $g_{\vec h}: \Z^t \to G^2/\Gamma^2$ is the sequence
\begin{equation}\label{gh-def} g_{\vec h}(\vec n) := (\{g_\lin(\vec h)\}^{-1} g_{\nonlin}(\vec n+\vec h) g_\lin(\vec n + \vec h) [g_\lin(\vec h)]^{-1}, g(\vec{n})).\end{equation}
One can check that $g_{\vec h}$ takes values in $G^{\Box} := G \times_{G_2} G$.
We may therefore replace \eqref{use-soon-1} by 
\begin{equation}\label{use-soon-1point1}
| \E_{\vec n \in [\vec N]} F^{\Box}_{\vec h} (g^{\Box}_{\vec h}(\vec n)\Gamma^{\Box})| \gg \delta^{O(1)}
\end{equation}
by restricting everything in that equation to an object on $G^{\Box}$, thus for instance $g^\Box_{\vec h}: \Z^t \to G^\Box$ is the map $g_{\vec h}: \Z^t \to G^2$ with range restricted to $G^\Box$.

By repeating the arguments in \cite[\S 7]{green-tao-nilratner}, $F_{\vec h}^{\Box}$ is invariant under $G_d^{\Delta} = \{(g_d,g_d) : g_d \in G_d\}$. 
Thus $F_{\vec h}^{\Box}$ descends to a function $\overline{F^{\Box}_{\vec h}}$ on $\overline{G^{\Box}} := G^{\Box}/G_d^{\Delta}$ and we may write \eqref{use-soon-1point1} as
\begin{equation}\label{use-soon-1point2}
| \E_{\vec n \in [\vec N]} \overline{F^{\Box}_{\vec h}} (\overline{g^{\Box}_{\vec h}}(\vec n)\overline{\Gamma^{\Box}})| \gg \delta^{O(1)},
\end{equation}
where $\overline{\Gamma^{\Box}} := \Gamma^{\Box}/(\Gamma \cap G_d^{\Delta})$.

We now have the key degree reduction proposition.

\begin{proposition}[Reduction in degree]
Define $(G^{\Box})_i := G_i \times_{G_{i+1}} G_i$ for $i = 1,\dots,d$. Then $(G^{\Box})_{\bullet}$ is a filtration on $G^{\Box}$ of degree $d$. Since $(G^{\Box})_d = G_d^{\Delta}$, it descends under quotienting by $G_d^{\Delta}$ to a filtration $\overline{(G^{\Box})}_{\bullet}$ of degree $d-1$ on $\overline{G^{\Box}}$. Each polynomial sequence $g_{\vec h}^{\Box}$ lies in $\poly(\Z^t,(G^{\Box})_{\bullet})$, and hence each reduced polynomial sequence $\overline{g_{\vec h}^{\Box}}$ lies in $\poly(\Z^t,(\overline{G^{\Box}})_{\bullet})$.
\end{proposition}

\begin{proof}  The first part of this proposition follows from \cite[Proposition 7.2]{green-tao-nilratner}; the remaining task is to show that
$g_{\vec h}^{\Box}$ lies in $\poly(\Z^t,(G^{\Box})_{\bullet})$.  Conjugating by $(\{g_\lin(\vec h)\}, 1_G) \in G^2$, it suffices to show that
$$
\vec n \mapsto (g_{\nonlin}(\vec n+\vec h) g_\lin(\vec n +\vec h) g_\lin(\vec h), g_{\nonlin}(\vec n) g_\lin(\vec n))$$
lies in $\poly(\Z^t,(G^{\Box})_{\bullet})$.  By arguing as in the proof of \cite[Proposition 7.2]{green-tao-nilratner},
$$
\vec n \mapsto (g_{\nonlin}(\vec n+\vec h), g_{\nonlin}(\vec n))$$
already lies in $\poly(\Z^t,(G^{\Box})_{\bullet})$, so it suffices to show that
$$
\vec n \mapsto (g_{\lin}(\vec n+\vec h) g_\lin(\vec h)^{-1}, g_{\nonlin}(\vec n))$$
lies in $\poly(\Z^t,(G^{\Box})_{\bullet})$.  We expand this as
$$
\vec n \mapsto (g(e_1)^{n_1+h_1} \ldots g(e_t)^{n_t+h_t} g(e_t)^{-h_t} \ldots g(e_1)^{-h_1}, g(e_1)^{n_1} \ldots g(e_t)^{n_t}).$$
Conjugating by $(g(e_1)^{h_1},1_G)$ and then factoring out $(g(e_1)^{n_1}, g(e_1)^{n_1})$, we may remove all factors of $g(e_1)$ here; iterating this procedure $t-1$ further times we obtain the claim.
\end{proof}

Now recall the following.

\begin{lem74rpt}
There is an $O(\delta^{-O(1)})$-rational Mal'cev basis $\mathcal{X}^{\Box} = \{X^{\Box}_1,\dots,X^{\Box}_{m^{\Box}}\}$ for $G^{\Box}/\Gamma^{\Box}$ adapted to the filtration $(G^{\Box})_{\bullet}$ with the property that $\psi_{\X^{\Box}}(x,x')$ is a polynomial of degree $O(1)$ with rational coefficients of height $\delta^{-O(1)}$ in the coordinates $\psi(x), \psi(x')$. With respect to the metric $d_{\mathcal{X}^{\Box}}$ we have $\Vert F^{\Box}_{\vec h} \Vert_{\Lip} \ll \delta^{-O(1)}$ uniformly in $\vec h$.
\end{lem74rpt}

By reducing the first $\overline{m^{\Box}} := m^{\Box} - m_d$ elements of $\mathcal{X}^{\Box}$ we obtain an $O(\delta^{-O(1)})$-rational Mal'cev basis $\overline{\mathcal{X}^{\Box}} = \{\overline{X^{\Box}_1},\dots,\overline{X^{\Box}_{\overline{m^{\Box}}}}\}$ for $\overline{G^{\Box}}/\overline{\Gamma^{\Box}}$ adapted to the filtration $\overline{(G^{\Box})}_{\bullet}$. With respect to the metric $d_{\overline{\mathcal{X}^{\Box}}}$ we have $\Vert \overline{F_{\vec h}^{\Box}} \Vert_{\Lip} \ll \delta^{-O(1)}$.

Since $\overline{(G^{\Box})}_{\bullet}$ has degree $d-1$ our inductive hypothesis is applicable and we conclude that for $\gg \delta^{O(1)} N_1 \ldots N_t$ values of $\vec h \in [-\vec N,\vec N]$ there is some horizontal character $\overline{\eta}_{\vec h} : \overline{G^{\Box}} \rightarrow \R/\Z$ with $0 < |\overline{\eta}_{\vec h}| \ll \delta^{-O(1)}$ and
\[ \Vert \overline{\eta}_{\vec h} \circ \overline{g^{\Box}_{\vec h}} \Vert_{C^{\infty}[\vec N]} \ll \delta^{-O(1)}.\] 
By pigeonholing in $\vec h$ we may assume that $\overline{\eta} = \overline{\eta_{\vec h}}$ is independent of $\vec h$. Writing $\eta : G^{\Box} \rightarrow \R/\Z$ for the horizontal character defined by $\eta(x) = \overline{\eta}(\overline{x})$, we see that $0 < |\eta| \ll \delta^{-O(1)}$ and that 
\begin{equation}\label{h-biases-2} \Vert \eta \circ g_{\vec h}^{\Box} \Vert_{C^{\infty}[\vec N]} \ll \delta^{-O(1)}.\end{equation}

Next recall the following.

\begin{lem75rpt}
We have a decomposition $\eta(g', g) = \eta_1(g) + \eta_2(g' g^{-1})$ for all $(g',g) \in G^\Box$, where $\eta_1 : G \rightarrow \R/\Z$ is a horizontal character on $G$, and $\eta_2 : G_{2} \rightarrow \R/\Z$ is a horizontal character on $G_{2}$ which also annihilates $[G,G_{2}]$. Furthermore we have $| \eta_1 |, |\eta_2 | \ll \delta^{-O(1)}$.
\end{lem75rpt}

Recalling the formula \eqref{gh-def} for $g_{\vec h}^{\Box}(\vec n) = g_{\vec h}(\vec n)$, we therefore have
\begin{align*} \eta(g_{\vec h}^{\Box} & (\vec n)) = \eta_1(g(n)) + \\ & \eta_2(\{g_\lin(\vec h)\}^{-1} g_{\nonlin}(\vec n+\vec h) g_\lin(\vec n + \vec h) [g_\lin(\vec h)]^{-1}
g_\lin(\vec n)^{-1} g_{\nonlin}(\vec n)^{-1}).\end{align*}
If $\eta_2$ is trivial, we thus have the conclusion of  \cite[Theorem 8.6]{green-tao-nilratner}, so we may assume henceforth that $\eta_2$ is non-trivial.  As $\eta_2$ vanishes on $[G,G_{2}]$, the above expression is equal to
\begin{align*}
\eta_1(g(n)) & + \eta_2(g_{\nonlin}(n+h)) - \\ & \eta_2(g_{\nonlin}(n)) + \eta_2( \{g_\lin(\vec h)\}^{-1}
g_\lin(\vec n+\vec h) [g_\lin(\vec h)]^{-1}
g_\lin(\vec n)^{-1}).\end{align*}
We simplify this expression further as follows.  As $\eta_2$ vanishes on $[G,[G,G]] \subset [G,G_2]$, we have
\begin{equation}\label{Eta2}
 \eta_2( a bc d) = \eta_2( a cb d ) + \eta_2([b,c])
 \end{equation}
whenever $a,b,c,d \in G$ are such that $abcd \in G_2$ (or equivalently $acbd \in G_2$).  From this we may conclude that the map $(b,c) \mapsto \eta_2([b,c])$ is a bilinear and antisymmetric map from $G \times G$ to $\R/\Z$.  Writing $[g_\lin(\vec h)]^{-1} = g_\lin(\vec h)^{-1} \{g_\lin(\vec h)\}$, we thus have
\begin{align*}
\eta_2& ( \{g_\lin(\vec h)\}^{-1} g_\lin(\vec n + \vec h) [g_\lin(\vec h)]^{-1} g_\lin(\vec n)^{-1} ) \\ &= \eta_2( [\{g_\lin(\vec h)\}^{-1}, g_\lin(\vec n+\vec h) g_\lin(\vec h)^{-1} ] ) + \\ & \qquad\qquad \eta_2( g_\lin(\vec n +\vec h) g_\lin(\vec h)^{-1} g_\lin(\vec n)^{-1} ) \\
&= \sum_{i=1}^t n_i \eta_2( [g(e_i), \{ g_\lin(\vec h)\}] ) +  \eta_2( g_\lin(\vec n +\vec h) g_\lin(\vec h)^{-1} g_\lin(\vec n)^{-1} ).
\end{align*}
Using \eqref{Eta2} we may write $\eta_2( g_\lin(\vec n +\vec h) g_\lin(\vec h)^{-1} g_\lin(\vec n)^{-1} )$ as
\begin{align*}
\eta_2& ( g(e_1)^{n_1+h_1} \ldots g(e_t)^{n_t+h_t} g(e_t)^{-h_t} \ldots g(e_1)^{-h_1} g(e_t)^{-n_t} \ldots g(e_1)^{-n_1} ) \\ & = \eta_2( [g(e_t)^{n_t}, g(e_{t-1})^{h_{t-1}} \ldots g(e_1)^{-h_1}] ) + \\ & \qquad 
\eta_2( g(e_1)^{n_1+h_1} \ldots g(e_{t-1})^{n_{t-1}+h_{t-1}} g(e_{t-1})^{-h_{t-1}} \ldots \\ & \qquad \qquad g(e_1)^{-h_1} g(e_{t-1})^{-n_{t-1}} \ldots g(e_1)^{-n_1} );
\end{align*}
iterating this $t$ times and using the bilinearity of $(b,c)\mapsto \eta_2([b,c])$, this becomes
$$ \sum_{1 \leq j < i \leq t} h_j n_i \eta_2( [g(e_i), g(e_j)] ).$$
Putting all this together, we thus have
$$
\eta(g_{\vec h}^{\Box}(\vec n)) = P(\vec n) + Q(\vec n+\vec h) - Q(\vec n) + \sum_{i=1}^t \sigma_i(\vec h) n_i$$
where
\begin{align*}
P(\vec n) &:= \eta_1(g(\vec n)) \\
Q(\vec n) &:= \eta_2(g_\nonlin(\vec n)) \\
\sigma_i(\vec h) &:= \eta_2( [g(e_i), \{ g_\lin(\vec h)\}] ) + \sum_{1 \leq j < i} h_j \eta_2( [g(e_i), g(e_j) ] ).
\end{align*}
Note that $P,Q$ are polynomial maps from $\Z^t$ to $\R/\Z$ of degree at most $d$, with 
$$ P(0)=Q(0)=Q(e_i)=0$$
for $i=1,\ldots,t$.  From \eqref{h-biases-2} we thus have
\begin{equation}\label{h-biases-3} 
\Vert P(\vec n) + Q(\vec n+\vec h) - Q(\vec n) + \sum_{i=1}^t \sigma_i(\vec h) n_i \Vert_{C^{\infty}[\vec N]} \ll \delta^{-O(1)}\end{equation}
for $\gg \delta^{-O(1)} N_1 \ldots N_t$ values of $\vec h \in [-\vec N, \vec N]$.

We now perform a multidimensional version of the arguments used to establish \cite[Lemma 7.6]{green-tao-nilratner}.  Let $1 \leq i,j \leq t$.  If we apply the second-order difference operator $\partial_{e_i} \partial_{e_j}$ to the expression inside the norm in \eqref{h-biases-3}, and then evaluate at $\vec n = 0$, we see that
$$ 
\Vert \partial_{e_i} \partial_{e_j} (P-Q)(0) + \partial_{e_i} \partial_{e_j} Q(\vec h) \Vert_{\R/\Z} \ll \delta^{-O(1)} N_i^{-1} N_j^{-1} .$$
Applying Proposition \ref{prop2.1} we conclude that
$$
\Vert \partial_{e_i} \partial_{e_j} (P-Q)(0) + \partial_{e_i} \partial_{e_j} Q \Vert_{C^\infty[\vec N], \delta^{-O(1)}} \ll \delta^{-O(1)} N_i^{-1} N_j^{-1}.
$$
In particular, we have
$$ \| \partial_{e_{i_1}} \ldots \partial_{e_{i_j}} Q(0) \|_{\R/\Z, \delta^{-O(1)}} \ll \delta^{-O(1)} N_{i_1}^{-1} \ldots N_{i_j}^{-1}$$
whenever $j \geq 3$.  By Taylor expansion, we thus have
$$ Q(\vec n) = \sum_{1 \leq i < j \leq t} \alpha_{ij} n_i n_j +\sum_{i=1}^t \alpha_{ii} \binom{n_i}{2} + R(\vec n)$$
where $\alpha_{ij} := \partial_{e_i} \partial_{e_j} Q(0)$ and $R$ obeys the bounds
$$ R(0) = R(e_i) = R(e_i+e_j) = 0$$
for $1 \leq i,j \leq t$ and
$$ \|R\|_{C^\infty[\vec N], \delta^{-O(1)}} \ll \delta^{-O(1)}.$$
Substituting this back into \eqref{h-biases-3}, we conclude that
\begin{align*}
\Vert P(\vec n) + & \partial_{\vec h} R(\vec n) + \sum_{i=1}^t (\sum_{j=1}^t \alpha_{ij} h_j + \sigma_i(\vec h)) n_i
+ \\ & \qquad\qquad \sum_{1 \leq i < j \leq t} \alpha_{ij} h_i h_j +\sum_{i=1}^t \alpha_{ii} \binom{h_i}{2}
 \Vert_{C^{\infty}[\vec N]} \ll \delta^{-O(1)}
 \end{align*}
for $\gg \delta^{O(1)} N_1 \ldots N_t$ values of $\vec h \in [-\vec N, \vec N]$.  If we apply $\partial_{e_i}$ and evaluate at $0$ for one such $\vec h$, recalling that $P(0)=0$, we conclude that
$$
\Vert P(e_i) + \partial_{e_i} \partial_{\vec h} R(\vec n) + (\sum_{j=1}^t \alpha_{ij} h_j + \sigma_i(\vec h)) 
 \Vert_{\R/\Z} \ll \delta^{-O(1)} N_i^{-1}
$$
which by the properties of $R$ implies that
\begin{equation}\label{pei}
\Vert P(e_i) + (\sum_{j=1}^t \alpha_{ij} h_j + \sigma_i(\vec h)) 
 \Vert_{\R/\Z,\delta^{-O(1)}} \ll \delta^{-O(1)} N_i^{-1}.
\end{equation}
By the pigeonhole principle, we may thus find $1 \leq q \ll \delta^{-O(1)}$ such that, for $\gg \delta^{O(1)} N_1 \ldots N_t$ values of $\vec h \in [-\vec N, \vec N]$, one has
$$
\Vert q P(e_i) + q (\sum_{j=1}^t \alpha_{ij} h_j + \sigma_i(\vec h)) 
 \Vert_{\R/\Z} \ll \delta^{-O(1)} N_i^{-1}
$$
for all $1 \leq i \leq t$.  By multiplying $\eta_1,\eta_2$ by $q$ if necessary we may assume in fact that $q=1$, thus
$$
\Vert P(e_i) + \sum_{j=1}^t \alpha_{ij} h_j + \sigma_i(\vec h) 
 \Vert_{\R/\Z} \ll \delta^{-O(1)} N_i^{-1}.
$$
We expand this as
$$
\Vert \beta_i + \sum_{j=1}^t \alpha'_{ij} h_j + 
\eta_2( [g(e_i), \{ g_\lin(\vec h)\}] ) \Vert_{\R/\Z} \ll \delta^{-O(1)} N_i^{-1},$$
where
$$ \beta_i := \eta_1(g(e_i))$$
and
$$ \alpha'_{ij} := \partial_{e_i} \partial_{e_j} Q(0) + 1_{j < i} \eta_2( [g(e_i), g(e_j)] ).$$
As the map $(b,c) \mapsto \eta_2([b,c])$ is bilinear, the map $c \mapsto \eta_2([g(e_i),c])$ is a homomorphism.   Thus there exists $\zeta_i \in \R^m$ such that
$$ \eta_2([g(e_i),x]) = \zeta_i \cdot \psi(x) \md{\Z}
$$
for all $x \in G$.  As discussed after \cite[(7.15)]{green-tao-nilratner}, all but the first $m_\lin$ coefficients of $\zeta_i$ are non-zero, and $|\zeta_i| \ll \delta^{-O(1)}$.  We thus have
\begin{equation}\label{sink}
\Vert \beta_i + \sum_{j=1}^t \alpha'_{ij} h_j + 
\zeta_i \cdot \{ \sum_{j=i}^t \gamma_j h_j \} \Vert_{\R/\Z} \ll \delta^{-O(1)} N_i^{-1},
\end{equation}
for $\gg \delta^{O(1)} N_1 \ldots N_t$ values of $\vec h \in [-\vec N, \vec N]$, where $\gamma_j := \psi( g(e_j) )$.

To handle this conclusion, we require the following multiparameter version of \cite[Claim 7.7]{green-tao-nilratner}.

\begin{lemma}\label{claim77gen}
Let $\beta, \alpha_1,\dots, \alpha_t \in \R$ and suppose that $\zeta, \gamma_1,\dots,\gamma_t \in \R^m$. Let $N, N_1,\dots, N_t$ be parameters and set $[-\vec{N}, \vec{N}] := \prod_{j = 1}^t [-N_j, N_j]$. Suppose that $|\zeta| \leq 1/\delta$ and that
\[ \Vert \beta + \sum_{j=1}^t \alpha_j h_j + \zeta \cdot \{ \sum_{j=1}^t \gamma_j h_j \} \Vert_{\R/\Z} \leq 1/\delta N\] for $\geq \delta N_1 \dots N_t$ values of $\vec{h} \in [-\vec{N}, \vec{N}]$. Then at least one of the following two statements is true:
\begin{itemize}
\item[(i)] There is $r \in \Z$, $0 <  r \ll \delta^{-O(1)}$, such that $\Vert r \zeta_l \md{\Z} \Vert_{\R/\Z} \ll \delta^{-O(1)}/N$ for $l = 1,\dots, m$;
\item[(ii)] There exists $k \in \Z^{m}$, $0 < |k| \ll \delta^{-O(1)}$, such that $\Vert k \cdot \gamma_j \Vert_{\R/\Z} \ll \delta^{-O(1)}/N_j$ for all $j = 1,\dots,t$.
\end{itemize}
The implied constant $O(1)$ may depend on $m$ and $t$.
\end{lemma}

 The proof of this statement goes along rather similar lines to that of \cite[Claim 7.7]{green-tao-nilratner}. However, the aforementioned proof was itself outsourced to no fewer than three earlier results from that paper, namely \cite[Proposition 3.1, Lemma 3.2, Proposition 5.3]{green-tao-nilratner}, which depend upon one another in sequence and which must now be formulated in a multparameter setting. On account of this undesirable state of affairs we give more details of these deductions in Appendix \ref{multiparameter-bracket}. 

Applying Lemma \ref{claim77gen} to \eqref{sink} for $i = 1,\dots, t$ (with $m = m_{\lin}$, $\beta = \beta_i$, $\alpha_i = \alpha'_{ij}$, $\zeta = \zeta_i$ and $N = N_i$), we see that either
\begin{itemize}
\item[(i)] There is $1 \leq r \ll \delta^{-O(1)}$ such that 
\[ \Vert r \zeta_{il} \md{\Z} \Vert_{\R/\Z} \ll \delta^{-O(1)}/N_i\] for $i = 1,\dots,t$ and $l = 1,\dots, m_{\lin}$,
where $\zeta_{il}$ is the $e_l$-component of $\zeta_i$; or
\item[(ii)] There exists $k \in \Z^{m_{\lin}}$, $0 < |k| \ll \delta^{-O(1)}$ such that $\Vert k \cdot \gamma_j \Vert_{\R/\Z} \ll \delta^{-O(1)}/N_j$ for $j = 1,\dots,t$.
\end{itemize}
Note that \emph{a priori} the application of Lemma \ref{claim77gen} gives, in option (i), a value of $r$ that depends on $i$. However by defining $r = r_1 \dots r_t$ we can eliminate this dependence.

If claim (ii) holds, then the horizontal character \[ \eta: x \mapsto k \cdot \psi(x) \md{\Z}\] is non-trivial and obeys the conclusion of Theorem \ref{mainthm-ind} by arguing exactly as in \cite[\S 7]{green-tao-nilratner}, so suppose instead that claim (i) holds. For each $1 \leq j \leq m$, let $\tau_j: G \to \R/\Z$ be the map
$$ \tau_j(x) := r \eta_2( [x, \exp(X_j)] ).$$
As in \cite[\S 7]{green-tao-nilratner}, $\tau_j$ is a horizontal character annihilating $G_2$ with $|\tau_j| \ll \delta^{-O(1)}$, and
$$ \tau_j(g(\vec n)) = r \sum_{i=1}^t n_i \zeta_{ij} \md{\Z},$$
so that $\|\tau_j\|_{C^\infty[\vec N]} \ll \delta^{-O(1)}$.  Thus, if any of the $\tau_j$ is non-zero, we obtain the conclusion of Theorem \ref{mainthm-ind}.  The only remaining case is if $\tau_j=0$ for all $j=1,\ldots,m$.  Arguing as in \cite[\S 7]{green-tao-nilratner}, this implies that $\eta_2$ annihilates $[G,G]$.  Thus the $\zeta_i$ vanish and $\alpha'_{ij} = \alpha_{ij}$, so that \eqref{sink} simplifies to
$$\Vert \beta_i + \sum_{j=1}^t \alpha_{ij} h_j  \Vert_{\R/\Z} \ll \delta^{-O(1)} N_i^{-1}$$
for all $1 \leq i \leq m$ and $\gg \delta^{O(1)} N_1 \ldots N_t$ values of $\vec{h} \in [-\vec N,\vec N]$.  Applying \cite[Lemma 3.2]{green-tao-nilratner} in each variable $h_j$ separately, we conclude that either\footnote{This possibility was omitted in previous versions of this erratum.} $N_i \ll \delta^{-O(1)}$ for some $i$, or else that
$$ \| \alpha_{ij}\|_{\R/\Z,\delta^{-O(1)}} \ll \delta^{-O(1)} N_i^{-1} N_j^{-1}$$
for all $1 \leq i,j \leq m$.  Arguing as in \cite[\S 7]{green-tao-nilratner} this implies that
$$ \| q \eta_2 \circ g_\nonlin \|_{C^\infty[\vec N]} \ll \delta^{-O(1)}$$
for some $1 \leq q \ll \delta^{-O(1)}$; by dilating $\eta,\eta_1,\eta_2$ by $q$ we may take $q=1$, thus
$$ \| \eta_2 \circ g_\nonlin \|_{C^\infty[\vec N]} \ll \delta^{-O(1)}.$$
The remainder of the proof of Theorem \ref{mainthm-ind} then proceeds by a routine adaptation of the last half of \cite[\S 7]{green-tao-nilratner} to the multi-dimensional case.

\section{Minor errata}\label{minor}

We take the opportunity to correct some further small points in \cite{green-tao-nilratner}.

\begin{itemize}
\item In the proof of \cite[Lemma 3.2]{green-tao-nilratner}, $k$ and $q$ are the same.
\item The invocation of \cite[Lemma 3.2]{green-tao-nilratner} in the proof of \cite[Proposition 5.3]{green-tao-nilratner} is only valid in the regime $\sup_i |\zeta_i| \leq \frac{\delta}{2(1+m)}$, due to the hypothesis $\eps \leq \delta/2$ in \cite[Lemma 3.2]{green-tao-nilratner}.  However, the conclusion $|\alpha| \ll_m \sup_i |\zeta_i| \delta^{-C}/N$ is trivially valid in the remaining case $\sup_i |\zeta_i| > \frac{\delta}{2(1+m)}$, due to the hypothesis $|\alpha| \leq 1/\delta N$. See also the proof of Proposition \ref{lem53multi} (a multiparameter version of \cite[Proposition 5.3]{green-tao-nilratner}) below.
\item After (7.6), $g_h$ should be $\tilde g_h$.
\item The last lines of the proof of Proposition 7.2 are valid for the case $j \geq 1$. For the $j=0$ case, one needs to replace $G_{j+1}$ by $G_2$, that is to say one needs to verify $g_i^{\binom{n+h}{i}}=g_i^{\binom{n}{i}} \md{G_2}$.  For $i \geq 2$ this is clear; for $i=0,1$ one can verify that $g_i$ is trivial since $g_2(0)=g_2(1) = \id_G$.
\item The proof of Lemma 7.8 is not correct as it stands, because we failed to check that $[G'_1, G'_1] \subseteq G'_2$. For this to hold we need $\ker \tilde\eta_2 \subset [G,G]$, which is not in general true. However, earlier arguments in Section 7 complete the analysis unless we are in this case, by the remarks at the bottom of page 512. Note that there is a further small misprint on page 512, line -3: this should state that $\eta_2$ annihilates $[G,G]$.
\end{itemize}

\appendix

\section{A lemma about bracket forms}

\label{multiparameter-bracket}

The purpose of this section is to indicate the proof of Lemma \ref{claim77gen}. As noted in the main text, this is a multiparameter version of \cite[Claim 7.7]{green-tao-nilratner}, and is proven in an analogous manner. Throughout this appendix we will have parameters $N_1,\dots, N_t$ and will write $[\vec{N}] := \prod_{j=1}^t [N_j]$, $[-\vec{N}, \vec{N}] := \prod_{j = 1}^t [-N_j, N_j]$. We will also occasionally use a further parameter $N$.

We begin with a multiparameter version of \cite[Proposition 3.1]{green-tao-nilratner}, proven in a very similar manner to that result.
\begin{proposition}\label{prop31multi}
Let $m \geq 1$, let $0 < \delta < \frac{1}{2}$, and let $\gamma_1,\dots, \gamma_t \in \R^m$. If the sequence $(\gamma_1 n_1 + \dots + \gamma_t n_t)_{\vec{n} \in [\vec{N}]}$ is not $\delta$-equidistributed in the torus $(\R/\Z)^m$ then there exists some $k \in \Z^m$ with $0 < |k| \ll \delta^{-O(1)}$ such that $\Vert k \cdot \gamma_j \Vert_{\R/\Z} \ll \delta^{-O(1)}/N_j$ for all $j = 1,\dots, t$.
\end{proposition}
\begin{proof}
We argue exactly as in the proof of \cite[Proposition 3.1]{green-tao-nilratner} (which is a standard quantitative Weyl equidistribution argument). With extremely minimal changes, we arrive at the conclusion that there is some $k \in \Z^m$, $0 < |k| \ll \delta^{-O(1)}$, such that 
\[ |\E_{\vec{n} \in [\vec{N}]} e(k \cdot (n_1 \gamma_1 + \dots + n_t \gamma_t))| \gg \delta^{O(1)}.\]
The average here factors into an average over each $n_j$ separately. Each of these averages is trivially bounded by $1$ and so we have
\[ |\E_{n_j \in [N_j]} e((k \cdot \gamma_j) n_j) | \gg \delta^{O(1)}\] for $j = 1,\dots, t$. Using the standard estimate
\[ |\E_{n \in [N]} e(nt)| \ll \min\bigg(1, \frac{1}{N \Vert t \Vert_{\R/\Z}}\bigg),\] the conclusion follows.
\end{proof}
Next we need a multiparameter version of \cite[Lemma 3.2]{green-tao-nilratner}. 
\begin{lemma}\label{lem32multi}
Suppose that $\vec{\alpha}\in \R^t$, $0 < \delta < \frac{1}{2}$ and $0 < \eps \leq \frac{1}{2}\delta$. Let $I \subset \R/\Z$ be an interval of length $\eps$ such that $\vec{\alpha} \cdot \vec{n} \in I$ for at least $\delta N_1 \dots N_t$ values of $\vec{n} \in [\vec{N}]$. Then there is $q \in \Z$, $0 < |q| \ll \delta^{-O(1)}$, such that $\Vert q \alpha_j \Vert_{\R/\Z} \ll \eps \delta^{-O(1)}/N_j$ for $j = 1,\dots, t$.
\end{lemma}
\begin{proof}
This follows easily from \cite[Lemma 3.2]{green-tao-nilratner} applied in each variable separately. Indeed for each $j$ there is a choice of the $n_{j'}$, $j' \neq j$, such that $\alpha_j n_j \in \tilde{I}$ for at least $\delta N_j$ values of $n_j \in [N_j]$. Here, $\tilde I$ is simply $I$ translated by $\sum_{j' \neq j} \alpha_{j'} n_{j'}$. Applying \cite[Lemma 3.2]{green-tao-nilratner}, we conclude that there is $q_j \in \Z$, $0 < |q_j| \ll \delta^{-O(1)}$, such that $\Vert q_j \alpha_j \Vert_{\R/\Z} \ll \eps \delta^{-O(1)}/N_j$. Setting $q := q_1 \dots q_t$ gives the result. 
\end{proof}

Now we establish a multiparameter version of \cite[Proposition 5.3]{green-tao-nilratner}.

\begin{proposition}\label{lem53multi}
Let $\beta, \alpha_1,\dots, \alpha_t \in \R$ and suppose that $\zeta, \gamma_1,\dots,\gamma_t \in \R^m$. Suppose that 
\begin{equation}\label{58new} \Vert \beta + \sum_{j=1}^t \alpha_j h_j + \zeta \cdot \{ \sum_{j=1}^t \gamma_j h_j \} \Vert_{\R/\Z} \leq 1/\delta N\end{equation} for $\geq \delta N_1 \dots N_t$ values of $\vec{h} \in [-\vec{N}, \vec{N}]$. Suppose that $|\alpha_j| \leq 1/\delta N_j$ and additionally that $|\zeta| \leq 1/\delta$. Then either $\sup_l |\zeta_l| \ll \delta^{-O(1)}/N$ or else there exists $k \in \Z^{m}$, $0 < |k| \ll \delta^{-O(1)}$, such that $\Vert k \cdot \gamma_j \Vert_{\R/\Z} \ll \delta^{-O(1)}/N_j$ for all $j = 1,\dots, t$.

\end{proposition}
\begin{proof}
If $\sup_l |\zeta_l| \leq 1/\delta N$ then the conclusion is immediate, so assume this is not the case. Then the assumption implies that 
\[ \Vert \beta + \sum_{j=1}^t \alpha_j h_j \Vert_{\R/\Z} \leq m\sum_l |\zeta_l| + \frac{1}{\delta N} \leq (m+1) \sup_l |\zeta_l|\] for $\geq \delta N_1 \dots N_t$ values of $\vec{h}$. 

If $\sup_l |\zeta_l| \leq \frac{\delta}{2(m+1)}$ then Lemma \ref{lem32multi} applies, and we conclude that there is $q$, $0 < |q| \ll \delta^{-O(1)}$, such that \begin{equation}\label{eq59}\Vert q \alpha_j \Vert_{\R/\Z} \ll \sup_l |\zeta_l| \delta^{-C}/N_j\end{equation} for $j = 1,\dots, t$. If $\sup_l |\zeta_l| > \frac{\delta}{2(m+1)}$ then a similar conclusion trivially holds due to the hypothesis $|\alpha_j| \leq 1/\delta N_j$.

In \cite[Proposition 5.3]{green-tao-nilratner} we argued that we could now set $q$ to $1$, but this is not possible here (basically because the statement of Proposition \ref{lem53multi} does not follow trivially in the event that \emph{some} $N_j$ is $\ll \delta^{-O(1)}$). Thankfully, this step was not strictly necessary in the proof of \cite[Proposition 5.3]{green-tao-nilratner} either. 

Split $[\vec{N}]$ into $O(\delta^{-O(1)})$ grids of the form $P_1 \times \dots \times P_t$, where each $P_j$ is an arithmetic progression with common difference $q$ and length between $N'_j$ and $2N'_j$, where $N'_j := c \delta^{C+1} N_j$ and $c > 0$ is to be specified later. By the pigeonhole principle, we can find one of these grids $I$ in which there are $\geq \delta |I|$ values of $\vec{h}$ such that \eqref{58new} holds. If $c$ is chosen sufficiently small, then by \eqref{eq59} we see that $\sum_{j=1}^t \alpha_j h_j$ does not vary by more than $\frac{\delta}{20} \sup_l |\zeta_l|$ as $\vec{h}$ varies over such a grid.

Now we follow the proof of \cite[Proposition 5.3]{green-tao-nilratner}, concluding that either $\sup_l |\zeta_l| \leq 20/\delta^2 N$ (in which case the proposition holds) or else $(\gamma_1 n_1 + \dots + \gamma_t n_t)_{\vec{n} \in I}$ is not $c\delta^2$-equidistributed on the torus $(\R/\Z)^m$. Writing $I = P_1 \times \dots \times P_t$ with $P_j = x_j + q[N''_j]$, where $N'_j \leq N''_j \leq 2 N'_j$, we see that $(q (\gamma_1 n'_1 + \dots + \gamma_t n'_t))_{\vec{n}' \in [\vec{N}'']}$ is not $c\delta^2$-equidistributed on the torus. Applying Proposition \ref{prop31multi} (with $\gamma_j$ replaced by $q \gamma_j$) concludes the proof.
\end{proof}

Finally, we are ready to deduce Lemma \ref{claim77gen}, the deduction being along very similar lines to that of \cite[Claim 7.7]{green-tao-nilratner}. 
\begin{proof}[Proof of Lemma \ref{claim77gen}.] First apply Proposition \ref{lem53multi} with $m' := m+1$, $\beta' = \beta$, $\zeta' := (\zeta, 1)$, $\gamma'_j := (\gamma_j, \alpha_j)$ and $\alpha'_j := 0$, noting that 
\[ \beta' + \sum_{j = 1}^t \alpha'_j h_j + \zeta' \cdot \{ \sum_{j=1}^t \gamma'_j h_j\} = \beta + \sum_{j=1}^t \alpha_j h_j + \zeta \cdot \{ \sum_{j=1}^t \gamma_j h_j \} \md{1}.\]
We deduce that either $|\zeta'_l| \ll \delta^{-O(1)}/N$ for all $l$, in which case (i) holds and we are done, or else there exists $k \in \Z^m$ and $r \in \Z$, not both zero and with $|k|, |r| \ll \delta^{-O(1)}$, such that $\Vert k \cdot \gamma_j + r \alpha_j \Vert_{\R/\Z} \ll \delta^{-O(1)}/N_j$ for $j = 1,\dots, t$. If $r = 0$ then option (ii) of Lemma \ref{claim77gen} holds and we are done, so assume that $r \neq 0$. Multiplying the assumption of the lemma through by $r$ we see that for $\geq \delta N_1 \dots N_t$ values of $\vec{h} \in [-\vec{N},\vec{N}]$ we have
\[ \Vert \tilde \beta + \sum_{j = 1}^t \tilde\alpha_j h_j + \tilde \zeta \cdot \{\sum_{j=1}^t \gamma_j h_j \} \Vert_{\R/\Z} \ll \delta^{-O(1)}/N,\] where $\tilde \beta := r \beta$, $\tilde \alpha_j := \{ k \cdot \gamma_j + r \alpha_j\}$ and $\tilde \zeta := r \zeta - k$. Note that $|\tilde \alpha_j| \ll \delta^{-O(1)}/N_j$. Thus we may apply Proposition \ref{lem53multi} once more to conclude that either $|\tilde \zeta_l| \ll \delta^{-O(1)}/N$ for $l = 1,\dots, m$, which implies (i), or else there is a nonzero $\tilde k \in \Z^m$ such that $\Vert \tilde k \cdot \gamma_j \Vert_{\R/\Z} \ll \delta^{-O(1)}/N_j$ for $j = 1,\dots, t$, which implies (ii).
\end{proof}


\begin{thebibliography}{99}


\bibitem[GT]{green-tao-nilratner} B.~J.~Green and T.~C.~Tao, \emph{The quantitative behaviour of polynomial orbits on nilmanifolds,} Ann. of Math. (2) \textbf{175} (2012), no. 2, 465--540.


\bibitem{fks}  D.~Fisher, B.~Kalinin, and R.~Spatzier, \emph{Global rigidity of higher rank Anosov actions on tori and nilmanifolds,} With an appendix by James F. Davis. J. Amer. Math. Soc. \textbf{26} (2013), no. 1, 167--198.

\bibitem{gorodnik-spatzier} A.~Gorodnik and R.~Spatzier, \emph{Exponential Mixing of Nilmanifold Automorphisms,} preprint.

\bibitem{gorodnik-spatzier2} A.~Gorodnik and R.~Spatzier, \emph{Mixing properties of $\Z^k$ actions on nilmanifolds,} preprint.



\bibitem{arithmetic-regularity} B.~J.~Green and T.~C.~Tao, \emph{An arithmetic regularity lemma, an associated counting lemma, and applications,} An irregular mind, 261--334, Bolyai Soc. Math. Stud., \textbf{21}, J\'anos Bolyai Math. Soc., Budapest, 2010.

\bibitem{gtz}  B.~J.~Green, T.~C.~Tao and T.~Ziegler, \emph{An inverse theorem for the Gowers $U^{s+1}[N]$-norm,} Ann. of Math. (2) \textbf{176} (2012), no. 2, 1231--1372.



\end{thebibliography}
\end{document}